\documentclass[11pt]{amsart} 

\usepackage[dvips]{graphics} 
\usepackage{color,amssymb,amsbsy,amsmath,amsfonts,amssymb,
amscd,epsfig,times,graphics}

\newtheorem{theo}{Theorem}
\newtheorem{prop}{Proposition}[section]
\newtheorem{lem}{Lemma}[section]
\newtheorem{cor}{Corollary}
\newtheorem{rem}{Remark}[section]
\theoremstyle{definition}
\newtheorem{defi}{Definition}[section]
 
\numberwithin{equation}{section}

\newcommand{\C}{\mathbb C} 
\newcommand{\R}{\mathbb R}

\newcommand{\partx}{\partial/\partial x}

\begin{document}
  
\begin{abstract}

Let $D=\{\rho < 0\}$ be a smooth relatively compact domain in an almost complex manifold $(M,J)$,   
where $\rho$ is a smooth defining function of $D$, strictly $J$-plurisubharmonic in a neighborhood of the closure $\overline{D}$ of $D$. We prove that $D$ has a connected boundary and is Gromov hyperbolic.
\end{abstract} 

\title[Gromov hyperbolicity of strongly pseudoconvex almost complex manifolds]
{Gromov hyperbolicity of strongly pseudoconvex almost complex manifolds} 

\author{Florian Bertrand and Herv\'e Gaussier}

\thanks{Research of the first author was supported by FWF grants AY0037721 and M1461-N25.}
\maketitle 

\section*{Introduction}
Complex Finsler geometry is an important branch of differential geometry, generalizing Hermitian geometry and carrying precious information on the geometry of the ambient complex manifold; its development goes back to the introduction of the Carath\'eodory pseudometric. The interest in complex Finsler geometry increased with the works of S.Kobayashi and his characterization of ample vector bundles, see \cite{kob}. We will focus in our paper on the Kobayashi metric, another well-known example of a (not necessarily smooth) complex Finsler pseudometric. One may refer to \cite{kob2} for a presentation of complex hyperbolic spaces. When the integrated pseudodistance associated to that metric is a distance, the corresponding metric space is called Kobayashi hyperbolic and the complex manifold carrying the distance inherits complex dynamical properties particularly adapted to study spaces of holomorphic maps. In the case of bounded domains in $\mathbb C^n$, curvature conditions on the boundary furnish information on the global behaviour of geodesics in the associated metric space, as in the celebrated paper of L.Lempert \cite{lem} for smooth strongly convex domains, or on their local boundary behavior, as in \cite{fo-ro, gra, ma} for strongly pseudoconvex domains.

Gromov hyperbolic spaces were introduced by M.Gromov \cite{gr2} in the eighties as geodesic metric spaces in which 
geodesic triangles are thin. See for instance \cite{bo-he-ko, bo-sc, co-de-pa, gh-ha} for different extensions of the theory. 
Trying to find examples of Gromov hyperbolic metric spaces or to characterize them received much attention. One could quote for instance the papers by Z.Balogh-S.Buckley \cite{bal-buc} in which the authors characterize Gromov hyperbolic domains in the Euclidean space $\mathbb R^n$ endowed with their quasi-hyperbolic distance, or by Y.Benoist in which the author characterizes Gromov hyperbolic convex domains for the Hilbert metric \cite{ben}. The Hilbert metric may be defined in the general setting of projective maps in real projective spaces, using the interval $]-1,1[$ as a gauje. The exact complex analogue provides the Kobayashi distance, where projective maps are replaced with holomorphic maps, the gauje being the unit disc endowed with the Poincar\'e metric. That similarity between the Hilbert metric and the Kobayashi metric was first observed by S.Kobayashi \cite{kob3}. The question of finding examples of Gromov hyperbolic domains, in $\mathbb C^n$, for their Kobayashi distance, is therfore natural. 
The Poincar\'e disc and as a generalization the unit ball in $\mathbb C^n$ endowed with its Kobayashi metric are well-known such examples.

The aim of the paper is to prove that every relatively compact, strictly pseudoconvex region in an almost complex manifold, endowed with the Kobayashi distance, is Gromov hyperbolic. In a first step we prove that such regions have a connected boundary using the Morse theory (Theorem 1). In a second step we prove the Gromov hyperbolicity of the region by studying  the large scale behaviour of geodesics (Theorem 2). That second part is in the vein of the results due to Z.Balogh-M.Bonk \cite{ba-bo} and L.Blanc-Centi \cite{bl}. For clarity and for the convenience of the reader we include a detailed proof, with the necessary adaptation in the (almost complex) manifold setting. We point out that a product of any two Kobayashi hyperbolic manifolds is never Gromov hyperbolic \cite{ba-bo}, explaining the strong pseudoconvexity condition on the domain. We also point out that a relatively compact strongly pseudoconvex domain in an almost complex manifold is not necessarily Kobayashi hyperbolic. One can prove in fact that a relatively compact strongly pseudoconvex domain in an almost complex manifold is complete hyperbolic if and only if it does not contain a Brody curve. However it is easy to see that a relatively compact strongly pseudoconvex region in an almost complex manifolds does not contain any Brody curve.

The two main results of the paper are the following:
\begin{theo}\label{theoconnect}
Let $D=\{\rho < 0\}$ be a smooth relatively compact domain in an almost complex manifold $(M,J)$,   
where $\rho$ is a smooth defining function of $D$, strictly $J$-plurisubharmonic in a neighborhood of the closure $\overline{D}$ of $D$. Then the boundary $\partial D$ of  $D$ is connected.
\end{theo}

\begin{theo}\label{theogro1}
Under the same assumptions as in Theorem \ref{theoconnect},  the domain $D$ endowed with its Kobayashi distance $d_{(D,J)})$ is Gromov hyperbolic. 
\end{theo}

\vskip 0,1cm
In Section 1 we present the necessary preliminaries. In Section 2 we prove Theorem~\ref{theoconnect}. Finally in Section 3 we prove Theorem~\ref{theogro1}.

\vskip 0,1cm
\noindent  {\it Acknowledgments.} The authors would like to thank G.Della Sala for helpful
discussions.

\section{Preliminaries}

\subsection{Almost complex manifolds and pseudoholomorphic discs}

An {\it almost complex structure} $J$ on a real smooth manifold $M$ is a $\left(1,1\right)$ tensor field
 which  satisfies $J^{2}=-Id$. We suppose that $J$ is smooth.
The pair $\left(M,J\right)$ is called an {\it almost complex manifold}. We denote by $J_{st}$
the standard integrable structure on $\C^{n}$ for every $n$.
A differentiable map $f:\left(M',J'\right) \longrightarrow \left(M,J\right)$ between two almost complex manifolds is said to be 
 {\it $\left(J',J\right)$-holomorphic}  if $J\left(f\left(p\right)\right)\circ d_{p}f=d_{p}f\circ J'\left(p\right),$ 
for every $p \in M'$. In case  $M'=\Delta$ is the unit disc in $\C$, such a map is called a {\it pseudoholomorphic disc}.  

\subsection{Strongly $J$-pseudoconvex domains}

Let $\rho$ be a smooth real valued function on
a smooth almost complex manifold  $\left(M,J\right).$
We denote by $d^c_J\rho$ the differential 
form defined by $d^c_J\rho\left(v\right):=-d\rho\left(Jv\right),$
where  $v$ is a section of $TM$. The {\it Levi form} of $\rho$ at a point $p\in M$ and a vector 
$v \in T_pM$ is defined by $\mathcal{L}_J\rho\left(p,v\right):=dd^c_J\rho(p)\left(v,J(p)v\right).$ We say that $\rho$ is {\it strictly $J$-plurisubharmonic} if 
$\mathcal{L}_J\rho\left(p,v\right)>0$ for any $p \in M$ and $v\neq 0 \in T_p M$. The boundary of a domain $D$ is {\it strongly $J$-pseudoconvex} if at any point 
$p \in \partial D$ there exists a  smooth strictly $J$-plurisubharmonic function $\rho$ defined in a neighborhood $U$ of $p$ in $M$ satisfying 
$\nabla \rho \neq 0$ on $\partial D \cap U$ and such that $D\cap U=\{\rho<0\}$. We say that a domain $D=\{\rho<0\}$ is a {\it strongly $J$-pseudoconvex region} 
in $(M,J)$ if $\rho$ is a smooth defining function of $D$,  strictly $J$-plurisubharmonic in a neighborhood of $\overline{D}$.

\subsection{Hypersurfaces of contact type}

Let $(V,\omega)$ be a symplectic manifold, namely $\omega$ is a closed, non-degenerate, skew symmetric 2-form on the smooth manifold $V$. Let $\Gamma$ be a hypersurface contained in $V$. We say that $\Gamma$ is of  {\it contact type} if there is a vector field $X$, defined near $\Gamma$, transverse to $\Gamma$ and pointing outwards, such that $d(i(X)\omega)=\omega$. The 1-form $\alpha=i(X)\omega$ is a  {\it contact form} on $\Gamma$ and it defines a  {\it contact structure} $\zeta=\ker \alpha$ on $\Gamma$.

Let $D$ be a relatively compact domain in an almost complex manifold.
The {\it complex tangent space} of the boundary $\partial D$ of $D$ is by definition $T^{J}\partial D:=T\partial D \cap J \partial D$.  
If $D$ is strongly $J$-pseudoconvex, the complex tangent space  $T^J\partial D$ is a contact structure. 
More precisely, let $\rho$ be a defining function of $\partial D$. Since $D$ is strongly $J$-pseudoconvex there exists a positive constant $c$ such that 
the function $\tilde \rho=\rho+c\rho^2$ is strictly $J$-plurisubharmonic in a neighborhood $U$ of  $\partial D$ and such that 
$\partial D = \{\tilde \rho = 0\}$ and  $D \cap U=\{\tilde \rho < 0\}$. 
Consider the one-form $d^c_J\tilde \rho$ and let $\alpha$ be its restriction to the tangent bundle $T\partial D$. It follows that 
$T^{J}\partial D={\rm Ker} \alpha$.
Due to the strict $J$-plurisubharmonicity of $\tilde \rho$, the 2-form $\omega:=dd^c_J\tilde\rho$ is a symplectic form on $U$ that tames $J$, $T^J\partial D$ is 
a contact structure  and $\alpha$ is a contact form for $T^J\partial D$. If  the boundary $\partial D$ of $D$ is connected, 
 a theorem due to W.L.Chow \cite{ch} states that any two points in $\partial D$ may be connected by a  $\mathcal{C}^1$ {\it horizontal curve},  i.e. a curve  $\gamma~: [0,1] \rightarrow \partial D$ satisfying  $\gamma'(s) \in T_{\gamma(s)}^J\partial D$ for every $s \in [0,1]$.
 This enables to define the {\it Carnot-Carath\'eodory metric } as follows (see \cite{bel, gr3}):
\begin{equation*}
\displaystyle d_{H}(p,q):= \inf \left\{ l(\gamma),
\gamma~: [0,1] \rightarrow \partial D \mbox{ }\mbox{ horizontal }\mbox{ }, \gamma(0)=p, \gamma(1)=q \right\},
\end{equation*}
where $l(\gamma)$ is  the {\it Levi length} of the horizontal curve $\gamma$ defined by 
$\displaystyle l(\gamma):= \int_0^1 \mathcal L_J\tilde\rho(\gamma(t),\gamma'(t))^{\frac{1}{2}}dt$.

\subsection{The Kobayashi pseudometric}
The existence of local pseudoholomorphic discs proved in \cite{ni-wo} 
enables to define the {\it Kobayashi pseudometric} $K_{\left(M,J\right)}$ for $p\in M$ and $v \in T_pM$:
$$K_{\left(M,J\right)}\left(p,v\right):=\inf 
\left\{\frac{1}{r}>0, u: \Delta \rightarrow \left(M,J\right) 
\mbox{  $J$-holomorphic }, u\left(0\right)=p, d_{0}u\left(\partx\right)=rv\right\},$$ 
and its integrated pseudodistance  $d_{\left(M,J\right)}$:
$$d_{\left(M,J\right)}\left(p,q\right): =\inf\left\{l_K(\gamma), \mbox{ }
\gamma: [0,1]\rightarrow M, \mbox{ }\gamma\left(0\right)=p, \gamma\left(1\right)=q\right\},$$
for $p,q \in M$, 
where where $l_K(\gamma)$ is  the {\it Kobayashi length} of a $\mathcal C^1$-piecewise smooth curve $\gamma$ defined by 
$\displaystyle l_K(\gamma):= \int_0^1 K_{\left(M,J\right)}\left(\gamma\left(t\right),\gamma'\left(t\right)\right)dt$.
The manifold $\left(M,J\right)$ is {\it Kobayashi hyperbolic} if $d_{\left(M,J\right)}$ is a distance.

Let $h$ be any $J$-hermitian metric in $M$ and denote by $\| \cdot \|_h$ the associated infinitesimal metric. A {\it Brody curve} is a nonconstant $J$-holomorphic curve from $\mathbb C$ to $M$ satisfying $\|df\|_h \leq 1$.
Notice that if $D$ is a relatively compact strongly $J$-pseudoconvex domain in $(M,J)$ or is compact, 
then $D$ contains no Brody curve if and only if $D$ is Kobayashi hyperbolic.

\subsection{Gromov hyperbolic spaces}

Let $(X,d)$ be a metric space. The {\it Gromov product} of two points $x,y \in X$ with respect to a basepoint $\omega \in X$ is defined by
$(x|y)_\omega:=\frac{1}{2}(d(x,\omega)+d(y,\omega)-d(x,y)).$ The Gromov product measures the failure of the triangle inequality to be an equality and
 is always nonnegative. The metric space $X$ is {\it Gromov hyperbolic} if there is a nonnegative constant $\delta$ such that for any
$x,y,z,\omega \in X$ one has:
\begin{equation}\label{3eqgrhy1}
(x|y)_\omega \geq \min((x|z)_\omega,(z|y)_\omega)-\delta.
\end{equation}
We point out that (\ref{3eqgrhy1}) can be also written as follows:
\begin{equation}\label{3eqgrhy2}
d(x,y)+d(z,\omega)\leq \max(d(x,z)+d(y,\omega),d(x,\omega)+d(y,z))+2\delta,
\end{equation}
for $x,y,z,\omega \in X$.


\begin{defi} 
Let $d$ and $d'$ be two metrics on $X$.
\begin{enumerate} 
\item $(X,d)$ and  $(X,d')$  are roughly isometric if there exists a positive constant 
$c$ such that for every $x,y \in X$:
$$d(x,y)-c \leq d'(x,y) \leq d(x,y)+c.$$ 
\item $(X,d)$ and  $(X,d')$ are quasi-isometric if there exist two positive constants 
$\lambda$ and $c$ such that for every $x,y \in X$:
$$\frac{1}{\lambda}d(x,y)-c \leq d'(x,y) \leq \lambda d(x,y)+c.$$
\end{enumerate}
\end{defi}

If  $(X,d)$ and $(X,d')$ are roughly isometric  then $(X,d)$ is Gromov hyperbolic if and only if  $(X,d')$ is Gromov hyperbolic. In case we consider  quasi-isometric spaces, the spaces need to be both geodesic. This is provided by  
the following theorem (see Theorem 12 p. 88 in \cite{gh-ha}) which is crucial for us:

\vskip 0,2cm
\noindent{\bf Theorem [Gh-dlH].}
{\sl
Let $(X,d)$ and $(X,d')$ be two quasi-isometric geodesic metric spaces. Then $(X,d)$ is Gromov hyperbolic if and only if $(X,d')$ is.
}

\vskip 0,2cm
We recall that a metric space $(X,d)$ is a {\it geodesic space} if any  two distinct points $x,y \in X$ can be joined by a {\it geodesic segment},
that is the image of an isometry $\gamma~: [0,d(x,y)] \rightarrow X$ with $\gamma(0)=x$ and $\gamma(d(x,y))=y$.

\section{Proof of Theorem~\ref{theoconnect}}

The proof of Theorem~\ref{theoconnect} is based on the Morse theory and more precisely on the fact that the topology of a domain defined by a Morse function can change only at critical points of the function (see Section I.3 "Homotopy Type in Terms of Critical Values" in \cite{mi} for more details). We recall that a {\it Morse function} is a smooth function which Hessian is non-degenerate at its critical points. The {\it Morse index} of a non-degenerate 
critical point  of such a function is the number of negative eigenvalues of its Hessian.
A level set $\{\rho=c\}$ containing a critical point of the function $\rho$ will be called a {\it critical set} and the corresponding value $c$ will be called a {\it critical value} of $\rho$.

Let $D=\{\rho<0\}$ be a relatively compact strongly $J$-pseudoconvex region.  After a small perturbation 
of $\rho$ that preserves  both the strict $J$-plurisubharmonicity of the defining function and the number of connected 
components of the boundary $\partial D=\{\rho=0\}$, we can suppose that $\rho$ is a Morse function. Hence $\rho$ 
has a finite number of critical points denoted by $p_1,\cdots,p_r$. We denote by $k_j$ the Morse index of 
the critical point $p_j$, $j=1,\cdots,r$.   
Moreover, we can always assume that every critical 
set contains only one critical point. We order the critical points in such a way such that the critical 
values $-c_j:=\rho(p_j)$ satisfy $-c_j>-c_{j+1}$ for $j=1,\dots,r-1$.
The number of connected components of   $\overline{D} \cap \{\rho= -t \}$ for $t \geq 0$ can change only at a 
critical level set. 
More precisely, it is sufficient to  describe the topology of 
$\overline{D} \cap \{\rho= -t \}$ in a neighborhood of a critical level set near a critical point.  
We denote by $C_{j,t} \in \mathbb{N}\cup\{\infty\}$ the number of connected components of 
the set $D\cap  \{\rho=-c_j+t\}$. 
\begin{lem}\label{lemcritic}
Under the above assumptions, we have for $\varepsilon>0$, sufficiently small:   
\begin{enumerate}
\item if $2\leq k_j \leq n$ then $C_{j,\varepsilon}=C_{j,-\varepsilon}$,
\item if $k_j=1$ then $ C_{j,\varepsilon} \leq C_{j,-\varepsilon}$,
\item for every critical point $p_j$ of Morse index  $k_j=0$, there is a neighborhood $U_j$ of  $p_j$ 
such that $\{\rho=-c_j+\varepsilon\} \cap U_j$ is connected and 
$\{\rho=-c_j-\varepsilon\} \cap U_j$ is empty. Moreover $k_r=0$ and $C_{r,\varepsilon}=1$.
\end{enumerate}   
\end{lem} 
Notice that according to this lemma, $C_{j,t}$ are finite for  $j=1,\dots,r-1$ and $t\in \R$.
In order to prove Lemma \ref{lemcritic}, we  need the following version of the Morse Lemma:
\begin{lem}\label{lemmorse}
Let $(M,J)$ be an almost complex manifold. Let $\rho$
be a strictly $J$-plurisubharmonic Morse function on $M$. Let $p$ be a critical point of $\rho$. Then there exist local (not necessarily holomorphic) 
coordinates $z=(z_1,\cdots,z_n)$, with $z_j=x_j+iy_j$, centered at $p$ and defined in a neighborhood $U_j$ of $p_j$, such that one has
\begin{equation}\label{eqmorse}
\rho(z) = \rho(0) + \sum_{j=1}^{n} x_j^2+\sum_{j=1}^{n} a_j y_j^2 +O(|z|^3)
\end{equation}
where $a_j=\pm 1$. In particular the Morse index of a critical point is smaller than $n$.    
\end{lem}

\begin{proof}[Proof of Lemma \ref{lemmorse}]
Let $p$ be a critical point of $\rho$. We use a normalization due to K.Diederich and A.Sukhov (see Lemma 3.2 and Proposition 3.5 in \cite{di-su}). In the associated coordinate system $w=(w_1,\cdots,w_n)$ centered at $p=0$, the Levi forms of $\rho$ at the origin, with respect to $J_{st}$ and $J$, coincide. In particular $\rho$ is strictly $J_{st}$-plurisubharmonic at the origin. Then
$$
\rho(w) = \rho(0) + \sum_{i,j=1}^n a_{ij} w_i \overline w_j
+ \Re e  \sum_{i,j=1}^n b_{ij} w_i w_j + O(|w|^3),
$$
where $A=(a_{ij})$ and $B=(b_{ij})$ are respectively  Hermitian and symmetric matrices.
Applying a linear transformation $w\mapsto \tilde{w}:=Lw$, we can reduce $A$ to the identity $A = I_n$.
Moreover we can make a unitary transformation $\tilde{w} \mapsto \tilde{z}:=U \tilde{w}$
preserving $\sum_{j=1}^{n} |\tilde{w}_j|^2$ and changing  the matrix $B$ into a diagonal matrix $U^t B U$ with nonnegative elements (see Lemma  7.2 in \cite{co-su-tu}). Then 
the expression of $\rho$ in the $\tilde{z}=(\tilde{z}_1,\cdots,\tilde{z}_n)$ coordinates reduces to
$$
\rho(\tilde{z}) = \rho(0) +  \sum_{j=1}^{n} |\tilde{z}_j|^2 + \sum_{j=1}^{n} \alpha_j \Re e \tilde{z}_{j}^2+ O(\vert \tilde{z} \vert^3),
$$
where $\alpha_j \geq 0$ for $j=1,\dots,n$. According to the non-degeneracy of the Hessian of $\rho$ at $0$, we have 
$\alpha_j\neq\pm 1$ for $j=1,\cdots,n$,
and thus we can set $x_j=(1+\alpha_j)^{\frac{1}{2}}\Re e \tilde{z}_j$ and $y_j=|1-\alpha_j|^{\frac{1}{2}}\Im m \tilde{z}_j$ which gives (\ref{eqmorse}).
\end{proof}

\begin{proof}[Proof of Lemma \ref{lemcritic}]
The content of Lemma \ref{lemcritic} is purely local. From now on we will denote by $U_j$ a connected neighborhood of  
the critical point $p_j$.
As pointed out it is sufficient to compute the number of connected components of the set $\{\rho = -c_j\pm \varepsilon\} \cap  U_j$.
 
Let $p_j$ be a critical point of $\rho$ of Morse index  $2 \leq k_j\leq n$. In a coordinate system centered at $p_j$ given by 
Lemma \ref{lemmorse}, $\rho$ can be written
$$
\rho(z) = -c_j + \sum_{l=1}^{n} x_l^2+\sum_{l=1}^{n} a_l y_l^2 +O(|z|^3)
$$
with $k_j$ negative coefficients among  $a_1\dots,a_n$, say for $I=\{l_1,\cdots,l_{k_j}\}$. For $t_\varepsilon>0$ sufficiently small, the set 
$\{\rho=-c_j+\varepsilon\} \cap \{\sum_{l \in I} y_{l}^2=t_{\varepsilon}\}\cap  U_j$ is a perturbation of a 
$(2n-k_j)$-sphere which varies smoothly with the real parameter $t_\varepsilon$. 
It follows that set $\{\rho=-c_j+\varepsilon\}\cap  U_j$ is connected for small positive $\varepsilon$. 
Moreover, for $t_\varepsilon'>0$  small enough, the set $\{\rho=-c_j+\varepsilon\} \cap \{\sum_{l=1}^{n} x_l^2+\sum_{l \notin I} y_{l}
^2=t_{\varepsilon}'\}\cap  U_j$  is a perturbation of a $k_j$-sphere, and thus the set $\{\rho=-c_j-\varepsilon\}\cap  U_j$ is connected  for small positive 
$\varepsilon$. This proves $(1)$.

Let $p_j$ be a critical point of $\rho$ of Morse index  $1$. In a coordinate system centered at $p_j$ provided by Lemma 
\ref{lemmorse}, $\rho$ has the form
$$
\rho(z) = -c_j + \sum_{l=1}^{n} x_l^2+\sum_{l\neq l_0} y_l^2 -y_{l_0}^2+O(|z|^3)
$$
for some $l_0$. The set $\{\rho=-c_j+\varepsilon\}\cap  U_j$ is connected whereas the set $\{\rho=-c_j-\varepsilon\}\cap  U_j$ is not. 
This  proves $(2)$.

Let $p_j$ be a critical point of $\rho$ of Morse index $0$. In a coordinate system centered at $p_j$ provided by Lemma \ref{lemmorse}, 
$\rho$ has the form
$$
\rho(z) = -c_j + \sum_{l=1}^{n} x_l^2+\sum_{l=1}^n y_l^2 +O(|z|^3),
$$ 
and so the set $\{\rho=-c_j+\varepsilon\} \cap U_j$ is connected and the set $\{\rho=-c_j-\varepsilon\}\cap U_j$ is empty. 
Moreover, since the domain $D$ is relatively compact then $k_r=0$. Finally since $p_r$ is the last  critical point of $\rho$, the level set 
$D\cap  \{\rho=-c_r+\varepsilon\}$ is connected and thus  $C_{r,\varepsilon}=1$.            
\end{proof}

\begin{proof}[Proof of Theorem \ref{theoconnect}] 
Assume by contradiction that the boundary $\partial D=\{\rho=0\}$ of $D$ has $m$ connected components
with $m \geq 2$. For $\varepsilon>0$ small enough, $D\cap \{\rho=-\varepsilon\}$ has $m$ connected components. We claim that  the number of connected components of $D\cap \{\rho= -t\}$ stays larger than or equal
to $m$ as $t\rightarrow c_r$, which contradicts the connectedness of the level sets nearby $\{\rho=-c_r\}$.
According to Lemma \ref{lemcritic}, critical points of $\rho$ of Morse index $k_j$ with $2\leq k_j \leq n$ preserve  the number of connected components of level sets nearby the critical level set $\{\rho=-c_j\}$. Critical points of Morse index $1$ may actually increase the number of connected components of level sets. 
Now assume that $p_j\neq p_r$ is a critical point of Morse index $0$ and denote by $D_j$ the connected 
component of $D\cap\{\rho \geq c_j-2\varepsilon\}$ containing $p_j$. The number of 
connected components of $\{\rho=-c_j-\varepsilon\}\cap D_j$ is positive otherwise  $D_j$ would be 
a connected component of $D$ different from $D$ contradicting the connectedness of  $D$.  
By part (3) of Lemma \ref{lemcritic}, the number of connected components of 
$\{\rho=-c_j+\varepsilon\}\cap D_j$ is larger than $1$ since the index of $p_j$ is $0$. Hence by 
connectedness of $D_j$, there exists a critical point $p_{j'} \in D_j$ with $j'<j$ such that  $C_{j',\varepsilon}<C_{j',-\varepsilon}$. This implies that $C_{j',\varepsilon} \leq C_{j,-\varepsilon}$.   

\end{proof}

\section{Proof of Theorem \ref{theogro1}}

We state a stronger version of Theorem \ref{theogro1}, namely:
\begin{theo}\label{theogro2} 
Let $D$ be a smooth relatively compact domain in an almost complex manifold $(M,J)$. We assume that the boundary $\partial D$ of $D
$ is connected and that $\partial D$ is strongly $J$-pseudoconvex. If $D$ does not contain any Brody $J$-holomorphic curve then 
$(D,d_{(D,J)})$ is Gromov hyperbolic.  
\end{theo}

Notice that a relatively compact domain $D$, defined by a global strictly $J$-plurisubharmonic function,  does not contain any Brody 
curves and thus is Kobayashi hyperbolic. Therefore Theorem \ref{theogro1} follows from Theorem \ref{theoconnect} and Theorem \ref{theogro2}. 
  
\subsection{Proof of Theorem \ref{theogro2}.} 
If $D$ does not contain any Brody $J$-holomorphic curve then we first point out that $(D,J)$ is Kobayashi hyperbolic. Indeed, assuming by contradiction that $(D,J)$ is not Kobayashi hyperbolic we may construct on $D$ a sequence of $J$-holomorphic discs with derivatives exploding at the centers. Since $\partial D$ is strongly $J$-pseudoconvex, it follows from estimates of the Kobayashi infinitesimal pseudometric (see \cite{ga-su}) that such discs do not approach the boundary $\partial D$ of $D$. Hence by a Brody renormalization process we construct a Brody curve contained in $D$, which is a contradiction.

We equip $M$ with an arbitrary smooth Riemannian metric and we denote by ${\rm dist}$ the associated distance.
For  $p \in D$ we define  a boundary projection map $\pi~: D \rightarrow \partial D$ by
${\rm dist}(p,\pi(p))={\rm dist}(p,\partial D)=:\delta(p).$ Notice that the map $\pi$ is uniquely defined near the boundary.  
Set $N_{\varepsilon}(\partial D):=\{q \in D, \delta(q)\leq \varepsilon\}$ where $\varepsilon$ is a sufficiently small positive real number such that $\pi$ is uniquely defined on $N_\varepsilon(\partial D)$ 
(see \cite{sp}, Chapter $9$, for the existence of such a tubular neighborhood). We emphasize that the construction of tubular neighborhoods using the exponential map gives a Riemannian analogue of
some Euclidean lemmas in \cite{ba-bo} (Lemma 2.1 and Lemma 2.2 in particular).
For points $p \in D \backslash N_{\varepsilon}(\partial D)$, $dist(p, \partial D)$ may be reached at different points. We pick up arbitrarily a point $\tilde p \in \partial D$ such that $dist(p,\partial D)=dist(p, \tilde p)$ and we set $\pi(p)=\tilde p$. The projection $\pi$ so defined is not intrinsic but this is not important for our goal.

We define the {\it height} of $p$ by 
$h(p):=\sqrt{\delta(p)}$. 
Then we define a metric  $g~:D\times D \rightarrow [0,+\infty)$ by:
\begin{equation*}
g(p,q):=2\log \left(\frac{d_H(\pi(p),\pi(q))+\max\{h(p),h(q)\}}{\sqrt{h(p)h(q)}}\right),
\end{equation*}
for $p,q \in D$ (see \cite{ba-bo}). It is important to notice that different choices of  a Riemannian metric and of a projection $\pi$ give
a different metric that coincides with $g$ up to an additive bounded term. This is not a matter of fact since we deal with roughly and quasi-isometric spaces.

Z.M.Balogh and M.Bonk proved in \cite{ba-bo} that the metric $g$ satisfies (\ref{3eqgrhy2}) and thus that the metric 
space $(D,g)$ is Gromov hyperbolic. Their proof in the Euclidean space extends directly to the almost complex case. Hence we have the following:
\begin{lem}\label{bb-lemma}
 The metric space $(D,g)$ is Gromov hyperbolic.
\end{lem}
However the space $(D,g)$ is not geodesic. In order to construct a geodesic Gromov hyperbolic metric space, 
we need to perturb the metric $g$. The following construction is due to L.Blanc-Centi \cite{bl} in $\mathbb R^{2n}$. For $p$ 
(resp. $q$) in $D$  
denoted by  $p_{\varepsilon}$ (resp. $q_{\varepsilon}$) the point on the fiber $\pi^{-1}(\pi(p))$ (resp. $\pi^{-1}(\pi(q))$) with 
height  $\sqrt{\varepsilon}$ 
and let
$
l_g(\gamma):=\sup_{0=t_0< t_1 < \cdots < t_n=1} \sum_{i=1}^{n}g(\gamma(t_{i-1}), ~\gamma(t_i)).
$
Then we define a new metric $d$ as follows:

\begin{equation*}d(p,q):=
\left\{
\begin{array}{lll} 
\inf \{l_g(\gamma), \gamma: [0,1]\rightarrow N_{\varepsilon}(\partial D) \mbox{smooth curve joining } p \mbox{ and } q\}~ \mbox{for}~ p, q \in N_{\varepsilon}(\partial D),\\
\\

d(p,q_{\varepsilon})+{\rm dist}(q, q_{\varepsilon})~\mbox{for} ~p \in N_{\varepsilon}(\partial D), q \in  D\setminus N_{\varepsilon}(\partial D), \\
\\
{\rm dist}(p,q)~\mbox{for} ~p, q \in D\setminus N_{\varepsilon}(\partial D)~  \mbox{such that}~ \pi(p)=\pi(q),\\
\\
{\rm
dist}(p,p_{\varepsilon})+d(p_{\varepsilon},q_{\varepsilon})+{\rm dist}(q,
q_{\varepsilon})
~\mbox{for} ~p, q \in D\setminus N_{\varepsilon}(\partial D) ~\mbox{such that} ~\pi(p)\neq \pi(q).\\
\end{array}
\right.
\end{equation*}

Then we have the following:
 \begin{prop}\label{prop}
 $(i)$ The metric space $(D,d)$ is Gromov hyperbolic.

$(ii)$ $(D,d)$ is geodesic.
 \end{prop} 
To prove Proposition \ref{prop} $(i)$ it is sufficient to prove that $(D,d)$ and $(D,g)$ are roughly isometric since $(D,g)$ is Gromov hyperbolic. The proof of $(ii)$ consists essentially in constructing geodesic curves joining any two points. The proof of Proposition \ref{prop} follows \cite{bl}. It was obtained there in case $M$ is the Euclidean space $\R^{2n}$ endowed with some almost complex structure. However the arguments remain valid for domains in an almost complex manifold. For convenience we include the key points of the proof.

\begin{proof}[Proof of Proposition \ref{prop}]
$(i)$ Since $(D,g)$ is Gromov hyperbolic according to Lemma \ref{bb-lemma}, we just need to prove that $(D,g)$ and  $(D,d)$ are roughly isometric. It is sufficient to prove that these spaces are roughly isometric near the boundary, namely
 that there is a positive constant $c$ such that for all  $ p,q \in N_{\varepsilon}(\partial D)$:
$$
g(p,q)-c \leq d(p,q) \leq g(p,q)+c.
$$
Restricting to $N_{\varepsilon}(\partial D)$ we define two types of curves in $(D,J)$:

- Normal curves are curves defined on the fibers of the projection. They are purely local objects 
since points considered belong to $N_{\varepsilon}(\partial D)$ and have the same boundary projection (see Lemma 1 in \cite{bl}). 

- Horizontal curves are curves  joining two points $p,q \in N_{\varepsilon}(\partial D)$ with same height $h(p)=h(q)$. They are defined as follows: 
since $(\partial D, d_H)$ is geodesic, we consider a geodesic curve $\alpha$ in $\partial D$ joining $\pi(p)$ and $\pi(q)$. 
For each $t$, we consider the point $\gamma(t) \in N_{\varepsilon}(\partial D)$ in the fiber $\pi^{-1}(\alpha(t))$ with height $h(p)=h(q)$. 
Then $\gamma$  defines a smooth horizontal curve in $N_{\varepsilon}(\partial D)$.

These two kinds of curves being defined for manifolds the proof that 
 $(D,g)$ and  $(D,d)$ are roughly isometric is then straightforward, see \cite{bl}.
Hence the metric space $(D,d)$ is Gromov hyperbolic.

\vskip 0,2cm
$(ii)$ The proof that $(D,d)$ is geodesic is achieved by studying the relative positions of any two points $p,q \in D$. That was first achieved in \cite{ba-bo}. By the definition of the metric 
$d$, it reduces to the two following cases.

\vskip 0,1cm
\noindent{\underline {Case 1}.} $p,q \in D\setminus N_{\varepsilon}(\partial D)$ satisfy $\pi(p)=\pi(q)$.  
Then $d$ coincides with the metric ${\rm dist}$ induced by a Riemannian metric. 
Since $(\overline{D\setminus N_{\varepsilon}(\partial D)}, {\rm dist})$ is compact, it is complete and thus according to the Hopf-Rinow Theorem, 
it is geodesic.

\vskip 0,1cm
\noindent{\underline{Case 2}.} $p,q \in N_{\varepsilon}(\partial D)$. One can prove that there exists a geodesic curve contained in $N_{\varepsilon}(\partial D)$, namely a curve $\gamma:[0,1] \rightarrow N_{\varepsilon}(\partial D)$ such that $l_g(\gamma)=d(x,y)$. This is the content of Lemma 4 in \cite{bl}.
\end{proof}

\begin{rem} It follows from the construction of the metric $g$ that the boundary $\partial_G D$ of $(D,g)$, as a Gromov hyperbolic space, may be identified with the boundary $\partial D$ of the domain $D$. The same holds for the boundary of $(D,d)$ as a Gromov hyperbolic space, since the spaces $(D,g)$ and $(D,d)$ are roughly isometric.
\end{rem}

The space $(D,d)$ being geodesic and Gromov hyperbolic by Proposition \ref{prop}, it remains to show that the metric space  $(D,d_{(D,J)})$ is geodesic and 
quasi-isometric to $(D,d)$, in order to complete the proof of Theorem~\ref{theogro2}. 
We will use the following precise estimates of the Kobayashi metric obtained in \cite{ga-su}: 

\vskip 0,2cm
\noindent{\bf Theorem [G-S].}
{\sl
Let $D$ be a relatively compact domain  in an almost complex manifold $(M,J)$. Assume that 
$\partial D=\{\rho=0\}$ where  
$\rho$ is  strictly $J$-plurisubharmonic in a neighborhood $U$ of $\partial D$. 
Then there exists a positive constant $C$  such that~:
\begin{equation}\label{eqest}
 \frac{1}{C}\left[\frac{|\partial_J\rho(p)(v - iJ(p)v)|^2}
{|\rho(p)|^2} + \frac{|v|^2}{|\rho(p)|}\right]^{1/2} \leq K_{(D,J)}(p,v) \leq  C\left[\frac{|\partial_J\rho(p)(v - iJ(p)v)|^2}
{|\rho(p)|^2} + \frac{|v|^2}{|\rho(p)|}\right]^{1/2}
\end{equation}
for every $p \in D\cap U$ and every $v \in T_pM$.
}

\vskip 0,2cm
We emphasize that as pointed out in Subsection $1.3$, the existence of such a function $\rho$ is ensured as soon as $D$ is strongly $J$-pseudoconvex. 
As a classical consequence of the lower estimate of the Kobayashi metric (\ref{eqest}), we obtain the completeness of  
the metric space $(D,d_{(D,J)})$. According to the Hopf-Rinow Theorem (see \cite{gr4} p. 9) the metric space $(D,d_{(D,J)})$ is geodesic.
Moreover we have:
\begin{prop}\label{prop-quasi}
 The metric spaces $(D,d_{(D,J)})$ and $(D,d)$ are quasi-isometric.
\end{prop}
The proof of Proposition \ref{prop-quasi} is essentially given by \cite{ba-bo} and remains valid in our setting. 
For seek of completeness  we include the key steps of the proof.
\begin{proof}[Proof of Proposition \ref{prop-quasi}]
Since $(D,d)$ and  $(D,g)$ are roughly isometric, it is sufficient to prove that $(D,d_{(D,J)})$ and $(D,g)$ are quasi-isometric, namely that there exists two positive constants $\lambda$ and $c$ such that for all  $p,q \in D$
\begin{equation}\label{eqqi}
\frac{1}{\lambda}g(p,q)-c \leq d_{(D,J)}(p,q) \leq \lambda g(p,q)+c.
\end{equation}
This is obtained by considering various cases depending on the relative positions of $p$ and $q$.

First notice that by integration, we obtain from (\ref{eqest}) some estimates of the Kobayashi length of
curves near the boundary. There are positive constants  $C_1$ and $C_2$ such that for any smooth curve
$\gamma: [0,1]\rightarrow  N_{\varepsilon}(\partial D)$ we have:
\begin{equation}\label{eqest1}
\frac{1}{C_1}\left |\log \frac{h(\gamma(1))}{h(\gamma(0))}\right |-C_2 \leq l_{K}(\gamma).
\end{equation}
Moreover in case $\gamma$ is a normal curve, the estimates are local since points on $\gamma$ have the same boundary projection and  
$\gamma([0,1])\subseteq N_{\varepsilon}(\partial D)$. Thus the situation is similar to the Euclidean case considered by \cite{ba-bo} and we have  
\begin{equation}\label{eqest2}
l_{K}(\gamma) \leq C_1\left |\log\frac{h(\gamma(1))}{h(\gamma(0))}\right |+C_2,
\end{equation}
increasing $C_1$ and $C_2$ if necessary.

\vskip 0,1cm
\noindent{\underline {Case 1}.} $p,q \in D\setminus N_{\varepsilon}(\partial D)$. Since $d_{(D,J)}$ and $g$ are non negative and  
uniformly bounded on 
$D\setminus N_{\varepsilon}(\partial D)$ we obtain (\ref{eqqi}). 

\vskip 0,1cm
\noindent{\underline {Case 2}.} $p \in N_{\varepsilon}(\partial D)$ and $q \in D\setminus N_{\varepsilon}(\partial D)$. By definition we have 
$$g(p,q):=2\log \left(\frac{d_H(\pi(p),\pi(q))+h(q)}{\sqrt{h(p)h(q)}}\right)$$
and thus
$$\log\frac{1}{h(p)}-c \leq g(p,q)\leq \log\frac{1}{h(p)}+c.$$ 

Since $d_{(D,J)}$ is uniformly bounded on $D\setminus N_{\varepsilon}(\partial D)$, we have
$$d_{(D,J)}(p,q)\leq d_{(D,J)}(p,p_\varepsilon)+d_{(D,J)}(p_\varepsilon,q)\leq d_{(D,J)}(p,p_\varepsilon)+c,$$
for some uniform positive constant still denoted by $c$. Here $p_\varepsilon$ is the unique point on $\pi^{-1}(\pi(p))$ with $h(p_\varepsilon)=\sqrt{\varepsilon}$ (or equivalently $\delta(p_\varepsilon)=\varepsilon$). Since $\pi(p)=\pi(p_\varepsilon)$,   we obtain with (\ref{eqest2})
$$d_{(D,J)}(p,p_\varepsilon)\leq C_1\log \frac{1}{h(p)}+C_2.$$
Therefore we obtain the right inequality of (\ref{eqqi}). 

Now consider a smooth curve $\gamma$ joining $p$ and $q$ and let $q'$ be the first point on the curve with $h(q')=\sqrt{\varepsilon}$. 
Let $\beta$ be the sub-curve of $\gamma$ joining $p$ and $q'$. According to (\ref{eqest1}) we get 
$$\frac{1}{C_1}\log\frac{1}{h(p)}-C_2 \leq l_K(\beta) \leq l_K(\gamma)$$
which implies, after taking the infimum over all possible curves $\gamma$ joining $p$ and $q$, that $\frac{1}{C_1}\log \frac{1}{h(p)}-C_2\leq d_{(D,J)}(p,q)$. Hence we obtain the left inequality of  (\ref{eqqi}).

\vskip 0,1cm
\noindent{\underline {Case 3}.} $p,q \in N_{\varepsilon}(\partial D)$ and ${\rm max}\{h(p),h(q)\}\geq d_H(\pi(p),\pi(q))$. 
Assume for instance that $h(p)\leq h(q)$. We have 
$$\log\frac{h(q)}{h(p)}-c \leq g(p,q)\leq \log\frac{h(q)}{h(p)}+c.$$

Denote by $p'$ the point in the fiber $\pi^{-1}(\pi(p))$ such that $h(p')=h(q)$. Then (\ref{eqest2}) provides
$$d_{(D,J)}(p,p')\leq C_1\log \frac{h(q)}{h(p)}+C_2.$$
Consider a horizontal curve $\gamma$ joining $p'$ and $q$.  Since $d_H(\pi(p'),\pi(q))=d_H(\pi(p),\pi(q))\leq {\rm max}\{h(p),h(q)\} \leq \varepsilon$, the situation 
is local and thus can be reduced to the Euclidean case. Hence following \cite{ba-bo}, we obtain
$$K_{(D,J)}(\gamma(t),\gamma'(t))^2\leq C \frac{\mathcal{L}_J\rho((\pi\circ \gamma)(t),(\pi\circ \gamma)'(t))}{\delta(q)}$$
and by integration
$$d_{(D,J)}(p',q)\leq C\frac{d_H(\pi(p),\pi(q))}{h(q)}.$$
Since $h(q)\geq d_H(\pi(p),\pi(q))$ we obtain $d_{(D,J)}(p',q)\leq C$. This gives the right inequality of (\ref{eqqi}).

The lower bound, and therefore the left inequality of (\ref{eqqi}), is easily obtained by applying (\ref{eqest1}).

\vskip 0,1cm
\noindent{\underline {Case 4}.} $p,q \in N_{\varepsilon}(\partial D)$ and ${\rm max}\{h(p),h(q)\} < d_H(\pi(p),\pi(q))$.
Assume for instance that $h(p)\leq h(q)$. We have 
$$2\log\frac{d_H(\pi(p),\pi(q))}{\sqrt{h(p)h(q)}}-c \leq g(p,q)\leq 2\log\frac{d_H(\pi(p),\pi(q))}{\sqrt{h(p)h(q)}}+c.$$

Set $h_0:={\rm min} \{\sqrt{\varepsilon},d_H(\pi(p),\pi(q))\}$ and denote by $p'$ (resp. $q'$) the point on the fiber $\pi^{-1}(\pi(p))$ 
(resp. $\pi^{-1}(\pi(q))$) with height $h_0$. As previously, one can show that
$$d_{(D,J)}(p,q)\leq C_1\log \frac{h_0}{\sqrt{h(p)h(q)}}+C_2.$$
by considering a normal curve joining $p$ and $p'$ (resp. $q$ and $q'$) and a horizontal curve joining $p'$ and $q'$. This gives the right inequality of  (\ref{eqqi}). 

Let $\gamma$ be a curve in $N_{\varepsilon}(\partial D)$ joining $p$ and $q$. Set $H:={\rm max}_{t \in [0,1]}h(\gamma(t))=h(\gamma(t_0))$ and denote by $\gamma_1$ (resp. $\gamma_2$) the sub-curve of $\gamma$ joining $p$ and $\gamma(t_0)$ (resp. $\gamma(t_0)$ and $q$). 

If $H\geq h_0$, the left inequality of (\ref{eqqi}) follows immediately by applying (\ref{eqest1}) to the sub-curves $\gamma_1$ and 
$\gamma_2$. In case $H< h_0$ we follow precisely \cite{ba-bo}, p. 521-523. Lower estimates of $d_{(D,J)}(p,q)$ are obtained by considering a dyadic decomposition of the curve $\gamma$ by height levels and some estimates relating 
the Kobayashi lengths of those pieces and the Carnot-Carath\'eodory lengths of the projected curve 
$\pi \circ \gamma$, known as the box ball estimates (see \cite{na-st-wa}). The crucial fact used 
in \cite{ba-bo} and that can be applied directly here is that the size of balls for the Carnot-Carath\'eodory distance can be approximated by 
Euclidean polydiscs. This ends the proof of Proposition~\ref{prop-quasi}. 

\end{proof}

Finally, Theorem [Gh-dlH] implies that the space $(D,d_{(D,J)})$ is Gromov hyperbolic which proves Theorem \ref{theogro2}. \qed

\subsection{A corollary of Theorem \ref{theogro2}} 

\begin{cor}\label{cor1}
Let $(V,\omega)$ be a symplectic manifold and let $D$ be a smooth relatively compact domain contained in $V$. Assume that $\partial D$ is connected and of contact type. Then there exists an almost complex structure $J$ on $V$, tamed by $\omega$, for which the following equivalence is satisfied:

$(D,d_{(D,J)})$ is Gromov hyperbolic if and only if $D$ does not contain Brody $J$-holomorphic curves.
\end{cor}
\begin{proof}

Let $\zeta=Ker (\alpha)$ be the contact structure on $\partial D$ (see Subsection 1.3). Since $\omega_{|\zeta}$ is nondegenerate, there exists an almost complex structure $J$ defined on a neighborhood of $\partial D$ and tamed by $\omega$, such that $\zeta$ is $J$-invariant and $d\alpha (v,Jv) > 0$ for all non-zero $v \in \zeta$. According to \cite{mcd} let $\rho: V \rightarrow \mathbb R$  be a smooth function such that $\partial D=\rho^{-1}(0)$ and $d\rho(X) > 0$ where $X$ is given by Subsection 1.3. Then for $v \in \zeta$ we have $Jv \in \zeta$ and $Ker (d^c_J\rho_{|\partial D})=Ker (\alpha_{|\partial D})$. This implies that there is a positive function $\mu$ defined in a neighborhood of $\partial D$ such that $dd^c_J\rho = \mu\, d\alpha$ on $\zeta$ (see Lemma 2.4 in \cite{mcd}). In particular $\partial D$ is strongly $J$-pseudoconvex. Since $\partial D$ is connected, we may apply Theorem \ref{theogro2} to conclude the proof of Corollary \ref{cor1}.
\end{proof}

\vskip 0,5cm
{\small
\noindent Florian Bertrand\\
Department of Mathematics, University of Vienna\\
Nordbergstrasse 15, Vienna, 1090, Austria\\
{\sl E-mail address}: florian.bertrand@univie.ac.at\\
\\
Herv\'e Gaussier\\
(1) UJF-Grenoble 1, Institut Fourier, Grenoble, F-38402, France\\
(2) CNRS UMR5582, Institut Fourier, Grenoble, F-38041, France\\
{\sl E-mail address}: herve.gaussier@ujf-grenoble.fr\\
} 
 
\end{document}